\documentclass[english]{amsart}

\usepackage{amsmath}
\usepackage{amssymb}
\usepackage{amsthm}
\usepackage{amscd}
\usepackage{babel}
\usepackage[latin1]{inputenc}
\usepackage{graphicx,color}

\newtheorem{theorem}{Theorem}
\newtheorem{cor}{Corollary}
\newtheorem{prop}{Proposition}
\newtheorem{con}{Conjecture}
\newtheorem{lem}{Lemma}

\theoremstyle{definition}

\newtheorem*{rem}{Remark}

\title{On power sums of matrices over a finite commutative ring}

\author{P. Fortuny }
\address{Departamento de Matem\'{a}ticas, Universidad de Oviedo\\ Avda. Calvo Sotelo, s/n, 33007 Oviedo, Spain}
\email{fortunypedro@uniovi.es}

\author{J.M. Grau}
\address{Departamento de Matem\'{a}ticas, Universidad de Oviedo\\ Avda. Calvo Sotelo, s/n, 33007 Oviedo, Spain}
\email{grau@uniovi.es}

\author{A.M. Oller-Marc\'{e}n}
\address{Centro Universitario de la Defensa de zaragoza\\ Ctra. Huesca s/n, 50090 Zaragoza, Spain}
\email{oller@unizar.es}

\author{I.F. R\'{u}a}
\address{Departamento de Matem\'{a}ticas, Universidad de Oviedo\\ Avda. Calvo Sotelo, s/n, 33007 Oviedo, Spain}
\email{rua@uniovi.es}

\begin{document}
\maketitle

\begin{abstract}
In this paper we deal with the problem of computing the sum of the $k$-th powers of all the elements of the matrix ring $\mathbb{M}_d(R)$ with $d>1$ and $R$ a finite commutative ring. We completely solve the problem in the case $R=\mathbb{Z}/n\mathbb{Z}$ and give some results that compute the value of this sum if $R$ is an arbitrary finite commutative ring $R$ for many values of $k$ and $d$. Finally, based on computational evidence and using some technical results proved in the paper we conjecture that the sum of the $k$-th powers of all the elements of the matrix ring $\mathbb{M}_d(R)$ is always $0$ unless $d=2$, $\textrm{card}(R) \equiv 2 \pmod 4$, $1<k\equiv -1,0,1 \pmod 6$ and the only element $e\in R \setminus \{0\}$ such that $2e =0$ is idempotent, in which case the sum is $\textrm{diag}(e,e)$.
\end{abstract}

\section{introduction}

For a ring $R$ we denote by $\mathbb{M}_d(R)$ the ring of $d\times d$ matrices over $R$. Now, given an integer $k\geq 1$ we define the sum
$$S_k^d(R):= \sum_{M\in \mathbb{M}_d(R)} M^k.$$
This paper deals with the computation of $S_k^d(R)$ in the case when $R$ is finite and commutative.

When $d=1$, the problem of computing $S_k^1(R)$ is completely solved only for some particular families of finite commutative rings. If $R$ is a finite field $\mathbb{F}_q$, the value of $S_k^1(\mathbb{F}_q)$ is well-known. If $R=\mathbb{Z}/n\mathbb{Z}$ the study of $S_k^1(\mathbb{Z}/n\mathbb{Z})$ dates back to 1840 \cite{VON} and has been completed in various works \cite{CAR,GMO,MOR2}. Finally, the case $R=\mathbb{Z}/n\mathbb{Z}[i]$ has been recently solved in \cite{for}. For those rings, we have the following result.

\begin{theorem}
Let $k\geq 1$ be an integer.
\begin{itemize}
\item[i)] Finite fields:
$$S_k^1(\mathbb{F}_q)=\begin{cases}
-1  , & \textrm{if $(q-1) \mid k $ };\\ 0  , &\textrm{otherwise}.
\end{cases}$$
\item[ii)] Integers modulo $n$:
$$S_k^1(\mathbb{Z}/n\mathbb{Z})=
\begin{cases}  \displaystyle{-\sum_{p \mid n, p-1 \mid k} \frac{n }{p }
    },
  & \textrm{if $k $ is even or $k=1$ or $n \not \equiv 0\pmod{4}$};\\
   \displaystyle{0
    }, & \textrm{otherwise}.
\end{cases}
$$
\item[iii)] Gaussian integers modulo $n$:
$$
S_k^1(\mathbb{Z}/n\mathbb{Z}[i])=
\begin{cases}  \frac{n}{2}(1+i),
  & \textrm{if $k>1$ is odd and $n\equiv 2\pmod{4}$};\\
  \displaystyle{-\sum_{p\in\mathcal{P}(k,n)} \frac{n^2}{p^2}}, & \textrm{otherwise}.
\end{cases}
$$
where $$
  \mathcal{P}(k,n):=\{ \textrm{prime $p$} : p \mid \mid n, p^2-1\mid k,
  p\equiv 3\pmod{4}\}
  $$
  and $p\mid\mid n$ means that $p\mid n$, but $p^2\nmid n$.
\end{itemize}
\end{theorem}

On the other hand, if $d>1$ the problem has been only solved when $R$ is a finite field \cite{BCL}. In particular, the following result holds.

\begin{theorem}
Let $k,d\geq 1$ be integers. Then $S_k^d(\mathbb{F}_q)=0$ unless $q=2=d$ and $1<k\equiv -1,0,1 \pmod{6}$ in which case $S_k^d(\mathbb{F}_q)=I_2$.
\end{theorem}

In this paper we deal with the computation of $S_k^d(R)$ with $d>1$ and $R$ a finite commutative ring. In particular Section \ref{pm} is devoted to completely determine the value of $S_k^d(R)$ in the case $R=\mathbb{Z}/n\mathbb{Z}$ (that we usually write as $\mathbb{Z}_n$). In Section \ref{ncm} we give some technical results regarding sums of non-commutative monomials over $\mathbb{Z}/n\mathbb{Z}$ which will be used in Section \ref{fcr} to compute $S_k^d(R)$ for an arbitrary finite commutative ring $R$ in many cases. Finally, we close the paper in Section \ref{cfw} with the following conjecture based on strong computational evidence

\begin{con} \label{conjf}Let $d>1$ and let $R$ be a finite commutative ring. Then $S_k^d(R)=0$ unless the following conditions hold:
\begin{enumerate}
\item $d=2$,
\item ${\rm card}(R) \equiv 2 \pmod{4}$ and $1<k\equiv -1,0,1 \pmod{6}$,
\item The unique element $e\in R \setminus \{0\}$ such that $2e =0$ is idempotent.
\end{enumerate}
Moreover, in this case
$$ S_k^d(R)=\begin{pmatrix} e & 0 \\ 0 & e \end{pmatrix}.$$
\end{con}

\section{Power sums of matrices over $\mathbb{Z}_n$}
\label{pm}
In what follows we will consider integers $n,d>1$. For the sake of simplicity, $M_n^d$ will denote the set of integer matrices with entries in the range $\{0,\dots,n-1\}$. Furthermore, for an integer $k\ge 1$, let $S^d_k(n)=\sum_{M\in M_n^d}M^k$. Our main goal in this section will be to compute the value of $S_k^d(n)$ modulo $n$. This is exactly the sum $S_k^d(\mathbb{Z}/n\mathbb{Z})$.

We start with the prime case. If $n=p$ is a prime, we have the following result \cite[Corollary 3.2]{BCL}

\begin{prop}
\label{p} Let $p$ be a prime. Then, $S_k^d(p)\equiv 0\pmod{p}$ unless $d=p=2$.
\end{prop}

Thus, the case $n=2$ must be studied separately. In fact, we have

\begin{prop}
\label{nd2}
$$S_k^2(2)\equiv\begin{cases}
0_2 \pmod{2}, & \textrm{if $k=1$ or $k\equiv 2,3,4\pmod{6}$};\\ I_2 \pmod{2}, &\textrm{if $1<k\equiv 0,1,5\pmod{6}$}.
\end{cases}$$
\end{prop}
\begin{proof}
For every $M\in M_n^2$ it holds that $M^2\equiv M^8\pmod{2}$. As a consequence $S_k^2(2)\equiv S_{k+6}^2(2)\pmod{2}$ for every $k>1$. Thus, the result follows just computing $S_k^2(2)$ for $1\leq k\leq 7$.
\end{proof}

Now, we turn to the prime power case. The following lemma is straightforward

\begin{lem}\label{desc}
Let $p$ be a prime. Then, any element $M$ in $M^d_{p^{s+1}}$ can be uniquely written in the form $A+p^sB$, where $A\in M^d_{p^s},B\in M^d_p$.
\end{lem}

Using this lemma we can prove the following useful result.

\begin{prop}\label{des}
Let $p$ be a prime. Then, $S_k^d(p^{s+1})\equiv p^{d^2} S_k^d(p^s)\pmod{p^{s+1}}$.
\end{prop}
\begin{proof}
By the previous lemma we have
\begin{equation}\label{e1}S_k^d(p^{s+1})=\sum_{M\in M_{p^{s+1}}^d} M^k=\sum_{A\in M_{p^s}^d}\sum_{B\in M_p^d} (A+p^sB)^k.\end{equation}
Using the non-commutative version of the binomial theorem we have that
$$(A+p^sB)^k\equiv A^k+p^s\sum_{t=1}^k A^{k-t}BA^{t-1}\pmod{p^{s+1}}.$$
Thus, combining this with (\ref{e1}) we obtain
\begin{align*}S_k^d(p^{s+1})&\equiv \sum_{B\in M_p^d}\left(\sum_{A\in M_{p^s}^d}A^k\right)+\sum_{t=1}^k\sum_{A\in M_{p^s}^d}A^{k-t}\left(p^s\sum_{B\in M_p^d}B\right)A^{t-1}\\
&\equiv p^{d^2}S_k^d(p^s)+\sum_{t=1}^k\sum_{A\in M_{p^s}^d}A^{k-t}\left(p^sS_1^d(p)\right)A^{t-1}\\
&\equiv p^{d^2}S_k^d(p^s)\pmod{p^{s+1}}
\end{align*}
because $S_1^d(p)\equiv 0\pmod{p}$ by Propositions \ref{p} and \ref{nd2} {(depending on whether $p$ is odd or not).}
\end{proof}

\begin{rem}
Note that Proposition \ref{des} implies that if $S_k^d(p^s)\equiv 0\pmod{p^s}$, then also $S_k^d(p^{s+1})\equiv 0\pmod{p^{s+1}}$.
\end{rem}

As a consequence we get the following result which extends Proposition \ref{p}.

\begin{cor}
\label{pp}
$S_k^d(p^s)\equiv 0\pmod{p^s}$ unless $d=p=2$ and $s=1$.
\end{cor}
\begin{proof}
If $p=d=2$, then Proposition \ref{p} implies that $S^2_k(4)\equiv 2^4 S^2_k(2)\equiv 0\pmod 4$, so the previous remark leads to $S^2_k(2^s)\equiv 0\ \pmod{2^s}$, for every $s>1$. On the other hand, if $d$ or $p$ is odd, then we know by Proposition 1 that $S^d_k( p )\equiv 0\pmod{p}$. Again, the remark gives us $S^d_k(p^s)\equiv 0$, by induction for all $s\ge 1$.
\end{proof}

In order to study the general case the following lemma will be useful. It is an analogue of \cite[Lemma 3 i)]{GOS}

\begin{lem}
\label{red}
If $m\mid n$, then $S_k^d(n)\equiv \left(\dfrac{n}{m}\right)^{d^2}S_k^d(m)\pmod{m}$.
\end{lem}
\begin{proof}
Given a matrix $M\in M_n^d$, let $M=\Big( m_{i,j}\Big)$ with $1\leq i,j\leq d$. Then,
\begin{align*}
S_k^d(n)&=\sum_{M\in M_n^d} M^k=\sum_{0\leq m_{i,j}\leq n-1}\Big( m_{i,j}\Big)^k\\
&\equiv \left(\dfrac{n}{m}\right)^{d^2} \sum_{0\leq m_{i,j}\leq m-1}\Big( m_{i,j}\Big)^k=S_k^d(m)\pmod{m}
\end{align*}
\end{proof}

Now, we can prove the main result of this section.

\begin{theorem}\label{teorfac}
The following congruence modulo $n$ holds:
$$S_k^d(n)\equiv\begin{cases}
\dfrac{n}{2}\cdot I_2, & \textrm{if $d=2$, $n\equiv 2\pmod{4}$ and $1<k\equiv 0,1,5\pmod{6}$};\\ 0_2, &\textrm{otherwise}.
\end{cases}$$
\end{theorem}
\begin{proof}
Let $n=2^sp_1^{r_1}\cdots p_t^{r_t}$ be the prime power decomposition of $n$.

If $1\leq i\leq t$, we have by Lemma \ref{red} and Corollary \ref{pp} that $$S_k^d(n)\equiv\left(\dfrac{n}{p_i^{r_i}}\right)^{d^2}S_k^d(p_i^{r_i})\equiv 0\pmod{p_i^{r_i}}.$$

On the other hand, using again Lemma \ref{red} we have that
$$S_k^d(n)\equiv \left(\dfrac{n}{2^s}\right)^{d^2}S_k^d(2^s)\pmod{2^s}.$$
Hence, Corollary \ref{pp} implies that $S_k^d(n)\equiv 0\pmod{2^s}$ unless $d=p=2$ and $s=1$.

To conclude, it is enough to apply Proposition \ref{nd2} together with the Chinese Remainder Theorem.
\end{proof}

The following corollary easily follows from Theorem \ref{teorfac} and it confirms the conjecture stated in the sequence A017593 from the OEIS \cite{OEIS}.

\begin{cor}
$S_n^2(n) \not\equiv  0 \pmod{n}$ if and only if $n \equiv 6 \pmod {12}$.
\end{cor}

As a further application of Theorem \ref{teorfac} application we are going to compute the sum of the powers of the Hamilton quaternions over $\mathbb{Z}/n\mathbb{Z}$.

\begin{prop}
For every $n \in \mathbb{N}$ and $l>0$, it holds that
$$\sum_{z \in \mathbb{Z}_n[i,j,k]} z^l=0.$$
\end{prop}
\begin{proof}
Since for all $z \in \mathbb{Z}_2[i,j,k]$ we have that $z^2\in \mathbb{Z}_2$, we deduce that $z^4=z^2$, and so it can be straightforwardly checked that
$$\sum_{z \in \mathbb{Z}_2[i,j,k]} z^l=0.$$

Now, if $s>1$, observing that
$$  \mathbb{Z}_{2^s}[i,j,k] \cong\left\{ \begin{pmatrix}
       a & b & c & d \\
       -b & a & -d & c \\
       -c & d & a & -b \\
       -d & -c & b & a \\
     \end{pmatrix}
   : a,b,c,d \in \mathbb{Z}_{2^s} \right\}
$$
we can adapt Lemma \ref{desc}, Proposition \ref{des} and Corollary \ref{pp} to inductively obtain that 
$$\sum_{z \in \mathbb{Z}_{2^s}[i,j,k]} z^l=0.$$

Finally, if $n=2^s m$ with $m$ odd we know \cite[Theorem 4]{GMiO} that
$$\mathbb{Z}_n[i,j,k]\cong \mathbb{Z}_{2^s}[i,j,k] \times \mathbb{Z}_m[i,j,k] \cong \mathbb{Z}_{2^s}[i,j,k] \times \mathbb{M}_2(\mathbb{Z}_m)$$
and the result follows from Theorem \ref{teorfac}.
\end{proof}

\section{Sums of non-commutative monomials over $\mathbb{Z}_n$}
\label{ncm}
We will now consider a more general setting. Let $r\geq 1$ be an integer and consider $w(x_1,\dots,x_r)$ a monomial in the non-commuting variables $\{x_1,\dots,x_r\}$ of total degree $k$. In this situation, we define the sum
$$S_w^d(n):=\sum_{A_1,\dots,A_r\in M_n^d} w(A_1,\dots,A_r).$$
Note that if $r=1$, then $w(x_1)=x_1^k$ and $S_w^d(n)=S_k^d(n)$ so we recover the situation from Section \ref{pm}. Thus, in what follows we assume $r>1$.

We want to study the value of $S_w^d(n)$ modulo $n$. To do so we first introduce two technical lemmas that extend \cite[Lemma 2.3]{BCL}.

\begin{lem}
\label{lemo}
Let $\tau\geq 1$ be an integer and let $\beta_i> 0$ for every $1\leq i\leq \tau$. If $p$ is an odd prime,
$$\sum_{x_1,\dots,x_{\tau}} x_1^{\beta_1}\cdots x_{\tau}^{\beta_{\tau}}\equiv\begin{cases}
(-p^{s-1})^\tau, & \textrm{if $p-1\mid \beta_i$ for every $i$};\\ 0, & \textrm{otherwise}.\end{cases}\pmod{p^s}$$
where the sum is extended over $x_1,\dots,x_{\tau}$ in the range $\{0,\dots,{p^s}-1\}$.
Also, if some $\beta_i=0$, then $\sum_{x_1,\dots,x_{\tau}} x_1^{\beta_1}\cdots x_{\tau}^{\beta_{\tau}}\equiv 0\pmod{p^s}$.
\end{lem}
\begin{proof}
It is enough to apply \cite[Lemma 3 ii)]{GOS} which states that
$$\sum_{x_i=0}^{p^s-1} x_i^{\beta_i}\equiv\begin{cases} -p^{s-1}, & if\ p-1 \mid \beta_i;\\ 0, & otherwise.\end{cases}\pmod{p^s}$$
for every $1\leq i\leq\tau$. Observe that, if $\beta_i=0$, then: $$\sum_{x_1,\dots,x_{\tau}} x_1^{\beta_1}\cdots x_{\tau}^{\beta_{\tau}}=\sum_{x_i}\sum_{x_j,j\not=i}x_1^{\beta_1}\cdots x_{i-1}^{\beta_{i-1}}x_{i+1}^{\beta_{i+1}}\cdots x_{\tau}^{\beta_{\tau}}\equiv 0\pmod{p^s}$$
\end{proof}

\begin{rem}\label{m1}
Observe that in the previous situation, if $\tau\geq 2$ and $s>1$, it easily follows that $\displaystyle \sum_{x_1,\dots,x_{\tau}} x_1^{\beta_1}\cdots x_{\tau}^{\beta_{\tau}}\equiv 0\pmod{p^s}$ regardless the values of $\beta_i\ge 0$.
\end{rem}

\begin{lem}
\label{lem2}
Let $\tau\geq 1$ be an integer and let $\beta_i> 0$ for every $1\leq i\leq \tau$. Then,
$$\sum_{x_1,\dots,x_{\tau}} x_1^{\beta_1}\cdots x_{\tau}^{\beta_{\tau}}\equiv\begin{cases}
1, & \textrm{if $s=1$};\\ 0, & \textrm{if $s>1$ and $\beta_i>1$ and odd for some $i$};\\ (-1)^{A}(2^{s-1})^B, & \textrm{if $s>1$ and $\beta_i=1$ or even for every $i$}\end{cases}\pmod{2^s}$$
where the sum is extended over $x_1,\dots,x_{\tau}$ in the range $\{0,\dots,{2^s}-1\}$, $A=\textrm{card}\{\beta_i:\beta_i=1\}$ and $B=\textrm{card}\{\beta_i:\beta_i\ \textrm{is even}\}$. Also, if some $\beta_i=0$, then $\sum_{x_1,\dots,x_{\tau}} x_1^{\beta_1}\cdots x_{\tau}^{\beta_{\tau}}\equiv 0\pmod{2^s}$.
\end{lem}
\begin{proof}
It is enough to apply \cite[Lemma 3 iii)]{GOS} which states that
$$\sum_{x_i=0}^{2^s-1} x_i^{\beta_i}\equiv\begin{cases} 2^{s-1}, & \textrm{if $s=1$ or $s>1$ and $\beta_1>1$ is even};\\ -1, & \textrm{if $s>1$ and $\beta_i=1$};\\ 0, & \textrm{if $s>1$ and $\beta_1>1$ is odd}.\end{cases}\pmod{p^s}$$
for every $1\leq i\leq\tau$. The proof of the case when some $\beta_i=0$ is identical to that of the previous lemma.
\end{proof}

As a consequence, we get the following results.

\begin{prop}\label{exp>1}
Let $p$ be an odd prime and let $s>1$ be an integer. Then,
$$S_w^d(p^s)\equiv 0\pmod{p^s}.$$
\end{prop}
\begin{proof}
Let $A_l=\big( a_{i,j}^l\big)_{1\leq i,j\leq d}$ for every $1\leq l\leq r$. Note that each entry in the matrix $S_w^d(p^s)$ is a homogeneous polynomial in the variables $a_{i,j}^l$. Observe also that these variables are summation indexes in the range $\{0,\dots,p^s-1\}$. Hence, the number of variables is $rd^2>2$ and, since $s>1$, the Remark \ref{m1} can be applied to the sum of its monomials, and the result follows.
\end{proof}

\begin{prop}
\label{prop2m1}
Let $s>1$ be an integer. Assume that one of the following conditions holds:
\begin{itemize}
\item[i)] $k\le rd^2$,
\item[ii)] $k> rd^2$ and $k+rd^2$ is even.
\end{itemize}
Then, $S_w^d(2^s)\equiv 0\pmod{2^s}$.
\end{prop}
\begin{proof}
Just like in the previous proposition each entry in the matrix $S_w^d(2^s)$ is a homogeneous polynomial in the $rd^2$ variables $a_{i,j}^l$. Hence, it is a sum of elements of the form
$$\sum_{a_{i,j}^l\in\mathbb{Z}_{2^s}} \prod (a_{i,j}^l)^{\beta_{i,j,l}}.$$

Observe that $\sum_{i,j,l} \beta_{i,j,l}=k$ so, if $k<rd^2$
it follows that some $\beta_{i,j,l}=0$, and so each monomial sum is 0 $\hbox{mod }2^s$ (because of Lemma \ref{m1}). Therefore, each entry in the matrix $S_w^d(p)$ is 0 $\pmod{2^s}$ in this case, as claimed.

Now, assume that $k\ge rd^2$ and $k+rd^2$ is even (in particular if $k=rd^2$). Due to Lemma \ref{lem2} an element $\displaystyle \sum_{a_{i,j}^l\in\mathbb{Z}_{2^s}} \prod (a_{i,j}^l)^{\beta_{i,j,l}}$ is 0 $\pmod{2^s}$ unless in one of its monomials the set of $rd^2$ exponents $\beta_{i,j,l}$ is formed by exactly $rd^2-1$ ones and 1 even value. But in this case $k=(rd^2-1)+2\alpha$ so $k+rd^2$ is odd, a contradiction. Consequently, each entry in the matrix $S_w^d(p)$ is also 0 $\pmod{2^s}$ in this case and the result follows.
\end{proof}

As Remark \ref{m1} and Lemma \ref{lem2} point out, the case $s=1$ must be considered separately. In this case, we have the following result.

\begin{prop}\label{exp1}
Let $p$ be a prime. Assume that one of the following conditions holds:
\begin{itemize}
\item[i)] $k<rd^2(p-1)$,
\item[ii)] $k$ is not a multiple of $p-1$.
\end{itemize}
Then, $S_w^d(p)\equiv 0\pmod{p}$.
\end{prop}
\begin{proof}
If $p=2$ condition ii) cannot hold and if condition i) holds, we can apply the same argument of the proof of the first part of Proposition \ref{prop2m1} to get the result.

Now, if $p$ is odd, again each entry in the matrix $S_w^d(p)$ is a homogeneous polynomial in the $rd^2$ variables $a_{i,j}^l$. Hence, it is a sum of elements of the form
$$\sum_{a_{i,j}^l\in\mathbb{Z}_p} \prod (a_{i,j}^l)^{\beta_{i,j,l}}.$$
We have that $\sum_{i,j,l} \beta_{i,j,l}=k$ so, if $k<rd^2(p-1)$ or if it is not a multiple of $p-1$ it follows that some $\beta_{i,j,l}$ is either $0$ or not a multiple of $p-1$. In either case the corresponding element is 0 $\pmod{p}$ due to Lemma \ref{lemo} and, consequently, each entry in the matrix $S_w^d(p)$ is also 0 $\pmod{p}$ as claimed.
\end{proof}

Observe that in the previous results we have considered sums of the form
$$S_w^d(p^s)=\sum_{A_1,\dots,A_r\in M_{p^s}^d} w(A_1,\dots,A_r),$$
where all the matrices $A_i$ belong to the same matrix ring $M_{p^s}^d$. The following proposition will be useful in the next section and deals with the case when the matrices $A_i$ belong to different matrix rings. First, we introduce some notation. Given a prime $p$, let
$$S_w^d(p^{s_1},\dots,p^{s_r}):=\sum_{A_i\in M_{p^{s_i}}^d} w(A_1,\dots,A_r).$$
If $s_1=\dots =s_r=s$, then $S_w^d(p^{s_1},\dots,p^{s_r})=S_w^d(p^s)$ and we are in the previous situation.

\begin{prop}
With the previous notation, if $s_1>1$, then
$$S_w^d(p^{s_1+1},p^{s_2},\dots,p^{s_r})\equiv p^{d^2}S_w^d(p^{s_1},p^{s_2},\dots,p^{s_r})\pmod{p^{s_1+1}}.$$
\end{prop}
\begin{proof}
Since $s_1>1$ we have that $2s_1>s_1+1$ so, due to Lemma \ref{desc}
\begin{align*}
S_w^d(p^{s_1+1},p^{s_2},\dots,p^{s_t})&=\sum_{\substack{A_1\in M_{p^{s_1+1}}^d\\ A_i\in M_{p^{s_i}}^d}} w(A_1,\dots,A_t)=\\&=\sum_{\substack{B\in M_{p^{s_1}}^d, C\in M_{p}^d\\ A_i\in M_{p^{s_i}}^d}} w(B+p^{s_1}C,A_2,\dots,A_r)\equiv \\ &\equiv \sum_{\substack{B\in M_{p^{s_1}}^d, C\in M_{p}^d\\ A_i\in M_{p^{s_i}}^d}} \left( w(B,A_2,\dots,A_r)+p^{s_1}\sum_{l} w_l(B,C,A_2,\dots,A_r)\right)=\\ &=
p^{d^2}S_w^d(p^{s_1},\dots,p^{s_r})+p^{s_1}\sum_l\sum_{\substack{B\in M_{p^{s_1}}^d, C\in M_{p}^d\\ A_i\in M_{p^{s_i}}^d}}w_l(B,C,A_2,\dots,A_r)\\ &\pmod{p^{s_1+1}}.
\end{align*}
Where $w_l(x,y,x_2,\dots,x_r)$ denotes the monomial $w(x_1,x_2,\dots,x_r)$ where the $l-th$ ocurrence of the term $x_1$ is substituted by $y$ and the remaining ones by $x$ (for instance, $w(x_1,x_2)=x_1^2x_2x_1$ gives us $w_1(x,y,x_2)=yxx_2x,w_2(x,y,x_2)=xyx_2x,w_3(x,y,x_2)=x^2x_2y$).

But, for every $l$, the monomial $w_l(B,C,A_2,\dots,A_r)$ contains $C$ only once and with exponent $1$. Hence,
$$\sum_{\substack{B\in M_{p^{s_1}}^d, C\in M_{p}^d\\ A_i\in M_{p^{s_i}}^d}}w_l(B,C,A_2,\dots,A_r)\equiv 0\pmod{p}$$
because $S_1^d(p)\equiv 0\pmod{p}$ and the result follows.
\end{proof}

The following corollary in now straightforward.

\begin{cor}\label{corvar}
Assume that $S_w^d(p^s)\equiv 0\pmod{p^s}$. Let us consider $s_1\geq s_2\geq\cdots\geq s_r=s$. Then,
$$S_w^d(p^{s_1},\dots,p^{s_r})\equiv 0\pmod{p^{s_1}}.$$
\end{cor}
\begin{proof}
Just apply the previous proposition repeatedly.
\end{proof}

\section{Power sums of matrices over a finite commutative ring}
\label{fcr}

In this section we will use the results from Section \ref{ncm} to compute $S_k^d(R)$ for an arbitrary finite commutative ring $R$ in many cases.

First of all, note that if $\textrm{char}(R)=n=p_1^{s_1}\cdots p_t^{s_t}$, then $ R \cong R_1 \times \cdots \times R_t$, where $\textrm{char}(R_i)=p_i^{s_i}$ and each $R_i$ is a subring of characteristic ${p_i^{s_i}}$ and, in particular, a $Z_{p_i^{s_i}}-$module. This allows us to restrict ourselves to the case when $\textrm{char}(R)$ is a prime power.

The simplest case arises when $R$ is a free $\mathbb{Z}_{p^{s}}-$module for an odd prime $p$.

\begin{prop}\label{free}
Let $p$ be an odd prime and let $R$ be a finite commutative ring of characteristic $p^s$, such that $R$ is a free $\mathbb{Z}_{p^{s}}-$module of rank $r$. Then,
\begin{itemize}
\item[i)] If $s>1$, $S_k^d(R)=0$ for every $k\geq 1$ and $d\geq 2$.
\item[ii)] If $s=1$, $S_k^d(R)=0$ for every $d\geq 2$ and $k$ such that either $k<rd^2(p-1)$ or $k$ is not a multiple of $p-1$.
\end{itemize}
\end{prop}
\begin{proof}
Note that under the previous assumptions and using Proposition \ref{exp>1} or Proposition \ref{exp1} (depending on whether $s>1$ or $s=1$), it follows that
$$\sum_{A_1,\dots,A_r\in M_{p^s}^d} (x_1 A_1+\cdots + x_r A_r)^k\equiv 0 \pmod{p^s}$$
because each entry of such a matrix is a polynomial in $x_1,\dots,x_r$ whose coefficients are 0 modulo $p^s$.

Consequently, for every $g_1,\dots,g_r\in R$ we have that
$$\sum_{A_1,\dots,A_r\in M_{p^s}^d} (g_1 A_1+\cdots + g_r A_r)^k=0.$$

Now, since $R$ is free of rank $r$ we can take a basis $g_1,\dots,g_r$ of $R$ so that $M^d_{p^s}=\{g_1A_1+\dots+g_rA_r | A_i\in M^d_{p^s}\}$. Therefore
$$S_k^d(R)=\sum_{A_1,\dots,A_r\in M_{p^s}^d} (g_1 A_1+\cdots + g_r A_r)^k.$$
This concludes the proof.
\end{proof}

If $p=2$, we have the following version of Proposition \ref{free}

\begin{prop}\label{free2}
Let $R$ be a finite commutative ring of characteristic $2^s$, such that $R$ is a free $\mathbb{Z}_{2^{s}}-$module of rank $r$. Then,
\begin{itemize}
\item[i)] If $s>1$, $S_k^d(R)=0$ for every $d\geq 2$ and $k$ such that $k\le rd^2$ or $k>rd^2$ with $k+rd^2$ even.
\item[ii)] If $s=1$, $S_k^d(R)=0$ for every $d\geq 2$ and $k$ such that either $k<rd^2$.
\end{itemize}
\end{prop}
\begin{proof}
The proof is similar to that of Proposition \ref{free}, using Proposition \ref{prop2m1} or Proposition \ref{exp1} depending on whether $s>1$ or $s=1$.
\end{proof}

\begin{rem}
Note that if $R$ is a finite commutative ring of characteristic $p^s$ and $s=1$, then $R$ is necessarily free. Consequently, to study the non-free case we may assume that $s>1$.
\end{rem}

Assume that elements $g_1,\dots,g_r$ form a minimal set of generators of a non-free $\mathbb{Z}_{p^{s}}-$module $R$. Since $R$ is non-free and $\textrm{char}(R)=p^s$, it follows that $r>1$ and also $s>1$. For every $i\in\{1,\dots, r\}$ let $1\leq s_i\leq s$ be minimal such that $p^{s_i}g_i=0$. Note that it must be $s_i=s$ for some $i$ and $s_j<s$ for some $j$. There is no loss of generality in assuming that $s=s_1\geq\cdots\geq s_r$ and at least one of the inequalities is strict. Note that $p^{s_1},\dots,p^{s_r}$ are the invariant factors of the $\mathbb{Z}-$module $R$. With this notation we have the following result extending Proposition \ref{free}.

\begin{prop}\label{nfp}
Let $p$ be an odd prime and let $R$ be a finite commutative ring of characteristic $p^s$, such that $R$ is a non-free $\mathbb{Z}_{p^{s}}-$module. Then,
\begin{itemize}
\item[i)] If $s_r>1$, $S_k^d(R)=0$ for every $k\geq 1$ and $d\geq 2$.
\item[ii)] If $s_r=1$, $S_k^d(R)=0$ for every $d\geq 2$ and $k$ such that either $k<rd^2(p-1)$ or $k$ is not a multiple of $p-1$.
\end{itemize}
\end{prop}
\begin{proof}
First of all, observe that
$$S_k^d(R)=\sum_{A_i\in M_{p^{s_i}}^d} (g_1 A_1+\cdots + g_r A_r)^k.$$
In both situations i) and ii) it follows that $S_w^d(p^{s_r})\equiv 0\pmod{p^{s_r}}$. Moreover, we are in the conditions of Corollary \ref{corvar}, so it follows that $S_w^d(p^s,p^{s_2},\dots,p^{s_r})\equiv 0\pmod{p^s}$. Consequently all the coefficients of the above sum are 0 modulo $p^s$ and the result follows.
\end{proof}

The corresponding result for $p=2$ is as follows.

\begin{prop}\label{nf2}
Let $R$ be a finite commutative ring of characteristic $2^s$, such that $R$ is a non-free $\mathbb{Z}_{p^{s}}-$module. Then,
\begin{itemize}
\item[i)] If $s_r>1$, $S_k^d(R)=0$ for every $d\geq 2$ and $k$ such that $k\le rd^2$ or $k>rd^2$ with $k+rd^2$ even.
\item[ii)] If $s_r=1$, $S_k^d(R)=0$ for every $d\geq 2$ and $k$ such that either $k<rd^2$.
\end{itemize}
\end{prop}
\begin{proof}
It is identical to the proof of Proposition \ref{nfp}.
\end{proof}

\section{Conjectures and further work}
\label{cfw}
Given a finite commutative ring $R$ of characteristic $n$, we have seen in the last section that $S_k^d(R)=0$ for many values of $k$, $d$ and $n$. In this section we present two conjectures based on strong computational evidence which, being true, would let us to give a general result about $S_k^d(R)$.

With the notation from the previous section, given an $r$-tuple of integers $\kappa=(k_1,\dots,k_r)$, we consider the set of monomials in the non-commuting variables $\{x_1,\dots, x_r\}$
$$ \Omega_{\kappa}:= \{w   : \textrm{deg}_{x_i}(w)=k_i,\ \textrm{for every $i$}\}.$$
The following conjectures are based on computational evidence.

\begin{con}\label{conj1}
With the previous notation, let $s_1\geq s_2 \geq \cdot\cdot\cdot \geq s_r$. Then
$$S_w^d(p^{s_1 },p^{s_2},\dots,p^{s_r})\equiv 0 \pmod{p^{s_1}},$$
unless $d=p=2$ and $s_i=1$ for all $i$.
\end{con}

\begin{con}\label{conj2}
If $p=2=d$ and $r>1$ then for every $\kappa \in \mathbb{N}^r$
$$\sum_{w \in \Omega_{\kappa}}   \sum_{A_i\in M_{2}^d} w(A_1,\dots,A_r) \equiv 0 \pmod{2}.$$
\end{con}

The next lemma extends Lemma \ref{red} in some sense. Its proof is straightforward.

\begin{lem} \label{lemg}
Let $R_1$ and $R_2$ be finite commutative rings, and let $R=R_1\times R_2$ be its direct product. Then
$$S_k^d(R)= (\textrm{card} (R_2)^{d^2} \cdot S_k^d( R_1),\textrm{card} (R_1)^{d^2} \cdot S_k^d( R_2))\in \mathbb{M}_d(R_1)\times \mathbb{M}_d(R_2)$$
\end{lem}

Now, the following proposition would follow from Conjectures \ref{conj1} and \ref{conj2}.

\begin{prop} \label{genp}
Let $R$ be a finite commutative ring of characteristisc $p^s$ for some prime $p$. Then $S_k^d(R)=0$ unless $d=2$, $R=\mathbb{Z}/2\mathbb{Z}$ and $1<k\equiv -1,0,1\pmod{6}$. Moreover, in this case $S_k^d(R)=I_2$.
\end{prop}
\begin{proof}
Assume that $\langle g_1\dots, g_r\rangle$ is a minimal set of generators of $R$ as $\mathbb{Z}_{p^s}$-module. Let $s=s_1\geq s_2 \geq \dots \geq s_r$ be integers such that the order of $g_i$ is $p^{s_i}$; i.e., $s_1,\dots, s_r$ are minimal such that $p^{s_i}g_i=0$.

In this situation,
$$ S_k^d(R)=\sum_{A_i \in M_{p^{s_i}}^d}(g_1 A_1 +...+g_r  A_r)^k = 0, $$
unless $d=p=2$, $s=r=1$ and $1<k\equiv -1,0,1 \pmod{6}$ due to Conjecture \ref{conj1}.

On the other hand, if $d=p=2$, $s=r=1$ and $1<k\equiv -1,0,1 \pmod{6}$ it follows that
$$S_k^2(R)=\sum_{A \in M_{2}^2}(g_1  A)^k = \begin{pmatrix} g_1^k & 0 \\ 0 &  g_1^k \end{pmatrix}  .$$
But since in this case $R=\{0,g_1\}$, there are only two possibilities: $g_1^2=g_1$ (and hence $R=\mathbb{Z}/2\mathbb{Z}$) or $g_1^2=0$ and the result follows.
\end{proof}

Finally, the next general result holds provided Conjectures \ref{conj1} and \ref{conj2} are correct. It is Conjecture \ref{conjf}, as stated in the introduction to the paper.

\begin{theorem} \label{teorfin}
Let $d>1$ and let $R$ be a finite commutative ring. Then $S_k^d(R)=0$ unless the following conditions hold:
\begin{enumerate}
\item $d=2$,
\item ${\rm card}(R) \equiv 2 \pmod{4}$ and $1<k\equiv -1,0,1 \pmod{6}$,
\item The unique element $e\in R \setminus \{0\}$ such that $2e =0$ is idempotent.
\end{enumerate}
Moreover, in this case
$$ S_k^d(R)=\begin{pmatrix} e & 0 \\ 0 & e \end{pmatrix}.$$
\end{theorem}
\begin{proof}
First, observe that if $\hbox{card}(R)\equiv 2\pmod{4},$ then $R$ has $2m$ elements, where $m$ is odd. Therefore, the $2-$primary component of the additive group $R$ has only two elements, and so there is a unique element $e\in R$ of additive order $2$.

Now, if $R$ is of characteristic $p^s$ for some prime, the result follows from the above proposition. Hence, we assume that $R$ has composite characteristic.
Let $R=R_1\times R_2$ with $R_1$ the zero ring or $\textrm{char}(R_1)=2^s$ and $\textrm{char}(R_2)$ odd. Due to Lemma \ref{lemg} and Proposition \ref{genp} it follows that $ S_k^d(R)=(\textrm{card}(R_2)^{d^2}\cdot S_k^d(R_1),0)$.

Now, $S_k^d(R_1)=0$ unless $d=2=p$, $R_1=\mathbb{Z}/2\mathbb{Z}$ and $1<k\equiv -1,0,1 \pmod 6$ in which case
$$ S_k^d(R)=\left( \begin{pmatrix} 1 & 0 \\ 0 &  1 \end{pmatrix}, \begin{pmatrix} 0 & 0 \\ 0 &  0 \end{pmatrix}\right)= \begin{pmatrix} e & 0 \\ 0 &  e \end{pmatrix},$$
where $e=(1,0) \in R_1\times R_2 $ is the only  idempotent of $R$ such that $2e =0$.
\end{proof}

\begin{rem}
Note that if, in addition, $R$ is unital then the element $e$ from the previous theorem is just $e=\frac{\textrm{card}(R)}{2}\cdot 1_R$. Also note that if $S_k^d(R)\neq 0$, then $ R \cong \mathbb{Z}/2\mathbb{Z}\times R_2$ with $\textrm{card}(R_2)$ odd or $R_2=\{0\}$.
\end{rem}

We close the paper with a final conjecture.

\begin{con}
Theorem \ref{teorfin} remains true if $R$ is non-commutative.
\end{con}

\end{document}